\documentclass{article}
\usepackage{amsmath,amsthm,amssymb,amsfonts}
\usepackage{verbatim}
\usepackage{graphicx}
\usepackage{stmaryrd} 
\usepackage{hyperref}

\usepackage[colorinlistoftodos]{todonotes}

\usepackage{tikz}
\usepackage{tkz-berge, tkz-graph}
\tikzset{vertex/.style={circle,draw,fill,inner sep=0pt,minimum size=1mm}}

\theoremstyle{plain}
\newtheorem{thm}{Theorem}
\newtheorem{lem}[thm]{Lemma}
\newtheorem{prop}[thm]{Proposition}

\newtheorem{remark}[thm]{Remark}
\newtheorem{question}[thm]{Question}

\theoremstyle{definition}
\newtheorem{definition}[thm]{Definition}

\newtheorem{exl}[thm]{Example}

\numberwithin{thm}{section}

\newcommand{\adj}{\leftrightarrow}
\newcommand{\adjeq}{\leftrightarroweq}

\def\Z{{\mathbb Z}}

\def\N{{\mathbb N}}
\begin{document}

\title{Approximate Fixed Point Property for
       Digital Trees and Products}
\author{Laurence Boxer
\thanks{
    Department of Computer and Information Sciences,
    Niagara University,
    Niagara University, NY 14109, USA;
    and Department of Computer Science and Engineering,
    State University of New York at Buffalo.
    Email: laurence.boxer@gmail.com 
}
}

\date{}
\maketitle{}
\begin{abstract}
We add to our knowledge of the approximate fixed point property
(AFPP) in digital topology.

We show that a digital image that is a tree has the AFPP.

Given two digital images $(X,\kappa)$ and $(Y,\lambda)$
that have the approximate fixed point property,
does their Cartesian product have the AFPP?
We explore conditions that yield an affirmative answer.
A general answer to this question is not known at the current
writing.
\end{abstract}

\section{Introduction}
The study of fixed points of continuous functions $f: X \to X$
has long captured the attention of researchers in many
areas of mathematics. It was introduced in digital
topology by A. Rosenfeld~\cite{Rosenfeld}. Rosenfeld showed
that even a digital image as simple as a digital interval
need not have a fixed point property (FPP), but does have an
``almost" or ``approximate" fixed point property (AFPP) (precisely
defined in~\cite{BoxKar12}). It was shown in~\cite{BoxKar12}
that among digital images, only singletons have the FPP;
perhaps as a consequence, attention shifted to the 
AFPP for digital images in such papers 
as~\cite{BoxerNP,BxAFPP,BEKLL,BoxKar12,CinKar,Han19}.
In this paper, we continue to study the AFPP for digital images;
in particular, for trees and for Cartesian products.

\section{Preliminaries}
Much of this section is quoted or 
paraphrased from the references,
especially~\cite{BxAFPP}.

We use $\Z$ to indicate the set of integers,
$\N$ for the set of natural numbers, and
$\N^*$ for the set of nonnegative integers.

\subsection{Adjacencies}
The $c_u$-adjacencies are commonly used.
Let $x,y \in \Z^n$, $x \neq y$, where we consider these points as $n$-tuples of integers:
\[ x=(x_1,\ldots, x_n),~~~y=(y_1,\ldots,y_n).
\]
Let $u \in \Z$,
$1 \leq u \leq n$. We say $x$ and $y$ are 
{\em $c_u$-adjacent} if
\begin{itemize}
\item there are at most $u$ indices $i$ for which 
      $|x_i - y_i| = 1$, and
\item for all indices $j$ such that $|x_j - y_j| \neq 1$ we
      have $x_j=y_j$.
\end{itemize}
Often, a $c_u$-adjacency is denoted by the number of points
adjacent to a given point in $\Z^n$ using this adjacency.
E.g.,
\begin{itemize}
\item In $\Z^1$, $c_1$-adjacency is 2-adjacency.
\item In $\Z^2$, $c_1$-adjacency is 4-adjacency and
      $c_2$-adjacency is 8-adjacency.
\item In $\Z^3$, $c_1$-adjacency is 6-adjacency,
      $c_2$-adjacency is 18-adjacency, and $c_3$-adjacency
      is 26-adjacency.
\end{itemize}

For $\kappa$-adjacent $x,y$, we write $x \adj_{\kappa} y$ or $x \adj y$ when $\kappa$ is understood.
We write $x \adjeq_{\kappa} y$ or $x \adjeq y$ to mean that either $x \adj_{\kappa} y$ or $x = y$.

We say $\{x_n\}_{n=0}^k \subset (X,\kappa)$ is a {\em $\kappa$-path} (or a {\em path} if $\kappa$ is understood)
from $x_0$ to $x_k$ if $x_i \adjeq_{\kappa} x_{i+1}$ for $i \in \{0,\ldots,k-1\}$, and $k$ is the {\em length} of the path.

A subset $Y$ of a digital image $(X,\kappa)$ is
{\em $\kappa$-connected}~\cite{Rosenfeld},
or {\em connected} when $\kappa$
is understood, if for every pair of points $a,b \in Y$ there
exists a $\kappa$-path in $Y$ from $a$ to $b$.

We define
\[ N(X,\kappa, x) = \{ y \in X \, | \, x \adj_{\kappa} y\},
\]
\[ N^*(X,\kappa, x) = \{ y \in X \, | \, x \adjeq_{\kappa} y\} =
   N(X,\kappa, x) \cup \{x\}.
\]

\begin{definition}
{\rm \cite{Berge}}
\label{normalDef}
Given digital images $(X,\kappa)$ and $(Y,\lambda)$,
the {\em normal product adjacency} $NP(kappa,\lambda)$ for
the Cartesian product $X \times Y$ is as follows. For
$x,x' \in X$, $y, y' \in Y$, we have
$(x,y) \adj_{NP(\kappa,\lambda)} (x',y')$ if
\begin{itemize}
    \item $x \adj_{\kappa} x'$ and $y=y'$, or
    \item $x=x'$ and $y \adj_{\lambda} y'$, or
    \item  $x \adj_{\kappa} x'$ and $y \adj_{\lambda} y'$.
\end{itemize}
\end{definition}

\subsection{Digitally continuous functions}
The following generalizes a definition of~\cite{Rosenfeld}.

\begin{definition}
\label{continuous}
{\rm ~\cite{Boxer99}}
Let $(X,\kappa)$ and $(Y,\lambda)$ be digital images. 
A function $f: X \rightarrow Y$ is 
{\em $(\kappa,\lambda)$-continuous} if for
every $\kappa$-connected $A \subset X$ we have that
$f(A)$ is a $\lambda$-connected subset of $Y$.
If $(X,\kappa)=(Y,\lambda)$, we say such a function is {\em $\kappa$-continuous},
denoted $f \in C(X,\kappa)$.
$\Box$
\end{definition}

When the adjacency relations are understood, we will simply say that $f$ is \emph{continuous}. Continuity can be expressed in terms of adjacency of points:
\begin{thm}
{\rm ~\cite{Rosenfeld,Boxer99}}
A single-valued function $f:X\to Y$ is continuous if and only if $x \adj x'$ in $X$ implies $f(x) \adjeq f(x')$. \qed
\end{thm}

Similar notions are referred to as {\em immersions}, 
{\em gradually varied operators}, and {\em gradually varied mappings}
in~\cite{Chen94,Chen04}.

Composition and restriction preserve continuity, in the sense of the following assertions.

\begin{thm}
{\rm \cite{Boxer99}}
\label{composition}
Let $(X,\kappa)$, $(Y,\lambda)$, and $(Z,\mu)$ be digital images.
Let $f: X \to Y$ be $(\kappa,\lambda)$-continuous and let
$g: Y \to Z$ be $(\lambda,\mu)$-continuous. Then
$g \circ f: X \to Z$ is $(\kappa,\mu)$-continuous.
\end{thm}

\begin{thm}
\label{restrictionThm}
{\rm \cite{CinKar}}
Let $(X,\kappa)$ and $(Y,\lambda)$ be digital images.
Let $f: X \to Y$ be $(\kappa,\lambda)$-continuous. 
\begin{itemize}
    \item Let $A \subset X$. Then $f|_A: A \to Y$ is
          $(\kappa,\lambda)$-continuous.
    \item $f: X \to f(X)$ is $(\kappa,\lambda)$-continuous.
\end{itemize}
\end{thm}

Given $X = \Pi_{i=1}^v X_i$, we denote throughout this paper the projection
onto the $i^{th}$ factor by $p_i$; i.e., $p_i: X \to X_i$ is defined by
$p_i(x_1,\ldots,x_v) = x_i$, where $x_j \in X_j$.

\begin{thm}
{\rm \cite{Han05}}
\label{projectionCont}
Given digital images $(X,\kappa)$ and $(Y,\lambda)$,
the projection maps $p_1$ and $p_2$ are
$(NP(\kappa,\lambda), \kappa)$-continuous and
$(NP(\kappa,\lambda), \lambda)$-continuous, respectively.
\end{thm}

\subsection{Approximate fixed points}
\label{approxPrelim}
Let $f \in C(X,\kappa)$
and let $x \in X$. We say
\begin{itemize}
    \item $x$ is a {\em fixed point} of $f$ if $f(x)=x$; 
    \item If $f(x) \adjeq_{\kappa} x$, then
          $x$ is an {\em almost fixed point}~\cite{Rosenfeld,Tsaur} or
          {\em approximate fixed point}~\cite{BEKLL} of 
          $(f,\kappa)$.
    \item A digital image $(X,\kappa)$ has the
          {\em approximate fixed point property} (AFPP)~\cite{BEKLL} if for every $g \in C(X,\kappa)$
          there is an approximate fixed point of $g$.
\end{itemize}

\begin{remark}
What we call the AFPP was denoted in~{\rm \cite{BxAFPP}} as the
$AFPP_S$ in order to distinguish it from its more general
version for multivalued continuous functions, denoted
$AFPP_M$. In this paper, we discuss only single-valued
continuous functions, so we use the simpler notation.
\end{remark}

\begin{thm}
{\rm \cite{BEKLL}}
\label{BEKLLisoThm}
Let $X$ and $Y$ be digital images such that $(X,\kappa)$ and
$(Y,\lambda)$ are isomorphic. If $(X,\kappa)$ has the AFPP,
then $(Y,\lambda)$ has the AFPP.
\end{thm}

\begin{thm}
{\rm \cite{BEKLL}}
\label{BEKLLretractionThm}
Let $X$ and $Y$ be digital images such that 
$Y$ is a $\kappa$-retract of $X$. If $(X,\kappa)$ has the AFPP, then $(Y,\kappa)$ has the AFPP.
\end{thm}

\section{Trees}
A {\em tree} is a triple $T=(X,\kappa,s)$, where $s \in X$ and
$(X,\kappa)$ is a connected graph that is acyclic, i.e., lacking any 
subgraph isomorphic to a cycle of more than 2 points. The vertex $s$ is the {\em root}. 
Given $x \adj_{\kappa} y$ in $X$, we say $x$ is the {\em parent of} $y$, 
and $y$ is a {\em child of} $x$, if $x \adj_{\kappa} y$ and
the unique shortest path from $y$ to the root contains $x$. 
Every vertex of the tree, except the root, has a unique parent vertex. A vertex, 
in general, may have multiple children. 
We define, recursively, a {\em descendant} of $x$ in a tree $T=(X,\kappa,r)$ as follows: $y \in X$
is a descendant of $x \in X$ if $y$ is a child of $x$ or $y$ is a descendant of a child of $x$.

We will use the following.

\begin{prop}
\label{builder}
Let $(X,\kappa)$ have the AFPP. Let $X'=X \cup \{x_0\}$,
where $x_0 \not \in X$, and let there be a $\kappa$-retraction
$r: X' \to X$ such that 
$N^*(X',\kappa,x_0) \subset N^*(X',\kappa,r(x_0))$.
Then $(X',\kappa)$ has the AFPP.
\end{prop}

\begin{proof}
Let $f \in C(X',\kappa)$. Then
$g=r \circ f|_X \in C(X,\kappa)$. Therefore, there
is an approximate fixed point $y \in X$ of $g$.
\begin{itemize}
    \item If $f(y) \in X$, then $f(y)=g(y) \adjeq_{\kappa} y$, 
          as desired.
    \item Otherwise, $f(y)=x_0$ and 
          $y \adjeq_{\kappa} g(y) = r(x_0)$. The continuity of $f$ 
          implies $f(g(y)) \adjeq_{\kappa} f(y)=x_0$, hence
          \[f(g(y)) \in N^*(X',\kappa,x_0) \subset 
           N^*(X',\kappa,r(x_0)) = N^*(X', \kappa,g(y)).
          \]
          So $g(y)$ is an approximate fixed point of $f$.
\end{itemize}
In either case, $f$ has an approximate fixed point. Since
$f$ was taken as an arbitrary member of $C(X', \kappa)$,
the assertion follows.
\end{proof}

\begin{thm}
A digital image $(T,\kappa)$ that is a tree has the AFPP.
\end{thm}
\begin{proof}
We argue by induction on $\#T$, the number of vertices in $T$.
The assertion is trivial for $\#T=1$. 

Suppose $k \in \N$ such that
the assertion is correct for all digital trees $T$ satisfying $\#T \le k$.
Now let $(T,\kappa)$ be a digital tree with $\#T = k+1$. Let
$v_0 \in T$ be a leaf of $T$, with $v_1 \in T$ as the parent of $v_0$.
Then $(T \setminus \{v_0\}, \kappa)$ is a digital tree of $k$ points. The function
$r: T \to T \setminus \{v_0\}$ defined by $r(v_0)=v_1$, 
$r(x) = x$ for $x \neq v_0$, is clearly a $\kappa$-retraction,
and $N^*(T,\kappa,v_0) = \{v_1\} \subset N^*(T,\kappa,r(v_0))$. 
It follows from the inductive hypothesis and 
Proposition~\ref{builder} that
$(T,\kappa)$ has the $AFPP_S$. This completes the induction.
\end{proof}

\section{Cartesian products}
In this section, we demonstrate an affirmative response to 
the following question.

\begin{question}
\label{prodQuestion}
{\rm \cite{BxAFPP}}
Let $X = \Pi_{i=1}^v [a_i,b_i]_{\Z}$, where for at least
2 indices~$i$ we have $b_i > a_i$. Does $(X,c_v)$ have
the AFPP?
\end{question}

Several authors have written that this question
was answered by Theorem~4.1 of~\cite{Rosenfeld}. However, 
it wasn't, as observed in~\cite{BxAFPP}:
\begin{quote}
    A. Rosenfeld's paper~\cite{Rosenfeld} states the following as its Theorem~4.1 (quoted verbatim).
\begin{quote}
    Let $I$ be a digital picture, and let $f$ be a continuous function from $I$
    into $I$; then there exists a point $P \in I$ such that $f(P)=P$ or is a neighbor
    or diagonal neighbor of $P$.
\end{quote}
Several subsequent papers have incorrectly
concluded that this result implies that $I$ with
some $c_u$ adjacency has the $AFPP_S$. 
By {\em digital picture} Rosenfeld means a digital cube, $I= [0,n]_{\Z}^v$.
By a ``continuous function" he means a $(c_1,c_1)$-continuous function;
by ``a neighbor or diagonal neighbor of $P$" he means a $c_v$-adjacent point.
\end{quote}

A partial solution to this problem is given in the following
(restated here in our terminology), which is Theorem~1 
of~\cite{Han19}. The ``proof" in~\cite{Han19} has multiple 
errors; a correct proof is given in~\cite{BxAFPP}.
\begin{thm}
\label{HanThm}
Let $X = [-1,1]_{\Z}^v$ and $1 \le u \le v$. Then
$(X,c_u)$ has the $AFPP_S$ if and only if $u=v$.
\end{thm}

We make use of the following.
\begin{thm}
{\rm \cite{BoxKar12}}
\label{BoxKarNP}
For $X \subset \Z^m$ and $Y \subset Z^n$,
$NP(c_m,c_n) = c_{m+n}$, i.e.,
given $x, x' \in X$, $y,y' \in Y$,
\[(x,y) \adj_{NP(c_m,c_n)} (x',y') \mbox{ if and only if }
(x,y) \adj_{c_{m+n}} (x',y').
\]
\end{thm}

\begin{remark}
It is shown in~{\rm \cite{BoxKar12}} that for $m \le M$,
$n \le N$, $m+n < M+N$, if $X \subset \Z^M$ and 
$Y \subset Z^N$, then we can have $NP(c_m,c_n) \neq c_{m+n}$.
\end{remark}

\begin{thm}
\label{extendToProd}
Let $(X,\kappa)$ be a digital image with the AFPP. Then the image
$(X \times [0,n]_{\Z}, NP(\kappa,c_1))$ has the AFPP.
\end{thm}

\begin{proof}
We argue by induction on $n$.

For $n=0$ we argue as follows. Since 
\[(X \times [0,0]_{\Z}, NP(\kappa,c_1))=(X \times \{0\}, NP(\kappa,c_1))
\]
is isomorphic to $(X,\kappa)$, it follows from Theorem~\ref{BEKLLisoThm} that
$(X \times [0,0]_{\Z}, NP(\kappa,c_1))$ has the
AFPP.

Now suppose $k \in \N^*$ and $(X \times [0,k]_{\Z}, NP(\kappa,c_1))$ has
the AFPP. To complete the induction, we must show that
$(X \times [0,k+1]_{\Z}, NP(\kappa,c_1))$ has the AFPP. Let
$r: X \times [0,k+1]_{\Z} \to X \times [0,k]_{\Z}$ be defined by
\[ r(x,t) = \left \{ \begin{array}{ll}
    (x,t) &  \mbox{if } 0 \le t \le k; \\
     (x,k) & \mbox{if } t = k+1.
\end{array} \right .
\]
Clearly, $r$ is $NP(\kappa,c_1)$-continuous and 
is a retraction.

Let $f \in C(X \times [0,k+1]_{\Z}, NP(\kappa,c_1))$.
Let $g: X \times [0,k]_{\Z} \to X \times [0,k]_{\Z}$
be defined by $g(x,t) = r \circ f \circ I(x,t)$, where
$I: X \times [0,k]_{\Z} \to X \times [0,k+1]_{\Z}$ is the
inclusion function. By the inductive hypothesis, $g$ has an
approximate fixed point; i.e., there exists
$p=(x_0,t_0) \in X \times [0,k]_{\Z}$ such that
\begin{equation}
\label{NPeq}
p \adjeq_{NP(\kappa,c_1)} g(p).
\end{equation}
\begin{itemize}
    \item If $f(p) \in X \times [0,k]_{\Z}$ then 
    \[ p \adjeq_{NP(\kappa,c_1)} g(p) =
       f(p),
    \]
          so $p$ is an approximate fixed point of $f$.
    \item Otherwise, we have that for some
          $x_1 \in X$, $f(p)=(x_1,k+1)$ and
          $g(p)=(x_1,k)$. 
          Let $p_1: X \times [0,k+1]_{\Z} \to X$
          and $p_2:  X \times [0,k+1]_{\Z} \to [0,k+1]_{\Z}$
          be the projections defined for $x \in X$, 
          $t \in [0,k+1]_{\Z}$ by
          \[ p_1(x,t) = x, ~~~p_2(x,t)=t.
          \]
          By Theorems~\ref{composition} and~\ref{projectionCont}, the functions
          $f \circ g$, $p_1 \circ f$, $p_1 \circ f \circ g$,
          $p_2 \circ f$, $p_2 \circ g$, and 
          $p_2 \circ f \circ g$ are all continuous.
          By continuity of $f$ 
          and~(\ref{NPeq}), $f(g(p)) \adjeq f(p)$,
          so 
          \begin{equation}
            \label{p1Adj}  
          p_1(f(g(p)) \adjeq_{\kappa} p_1(f(p)) = x_1 = p_1(g(p))
          \end{equation}
          and 
          $p_2(f(g(p))) \adjeq_{c_1} p_2(f(p))=k+1$, so
          $p_2(f(g(p)) \in \{k,k+1\}$, hence
          \begin{equation}
          \label{p2Adj}
          p_2(f(g(p)) \adjeq_{c_1} p_2(g(p)).
          \end{equation}
          By~(\ref{p1Adj}) and~(\ref{p2Adj}), $g(p)$ is an approximate fixed point of $f$.
\end{itemize}
In either case, $f$ has an approximate fixed point. This
completes the induction argument.
\end{proof}

\begin{lem}
\label{NPlemma}
Let $(X,\kappa)$ be a digital image. Consider
$(Y,c_v)$, where $Y=[0,n]_{\Z}^v$. 
For $X \times Y \times [0,n]_{\Z}$,
$NP(NP(\kappa,c_k)),c_1) = NP(\kappa, c_{k+1})$.
\end{lem}

\begin{proof}
Let $x,x' \in X$, $y,y' \in Y$, $t,t' \in [0,n]_{\Z}$, where
\[ y=(y_1,\ldots, y_v),~~~y'=(y_1',\ldots, y_v'),
\]
$y_i,y_i' \in [0,n]_{\Z}$ for $i = 1,\ldots, v$, such that
$(x,y,t) \neq (x',y', t')$.
Then
\[ (x,y,t) \adj_{NP(NP(\kappa,c_v),c_1)} (x',y',t') 
~~~\mbox{ if and only if }
\]
\[ (x,y) \adjeq_{NP(\kappa,c_v)} (x',y')
   ~~~\mbox{and}~~~ t \adjeq_{c_1} t' ~~~\mbox{ if and only if }
\]
\[ x \adjeq_{\kappa} x' ~~~\mbox{and}~~~y \adjeq_{c_v} y'
   ~~~\mbox{and}~~~t \adjeq_{c_1} t' ~~~\mbox{ if and only if }
   \]
\[ x \adjeq_{\kappa} x' ~~~\mbox{and}~~~(y,t) \adjeq_{NP(c_v,c_1)}(y',t')~~~\mbox{ if and only if }
\]
(by Theorem~\ref{BoxKarNP})
\[ x \adjeq_{\kappa} x' ~~~\mbox{and}~~~(y,t) \adjeq_{c_{v+1}}(y',t')~~~\mbox{ if and only if }
\]
$(x,y,t) \adjeq_{NP(\kappa,c_{v+1})}(x',y',t')$. The
assertion is established.
\end{proof}

\begin{thm}
\label{extendProdCube}
Let $(X,\kappa)$ be a digital image with the AFPP. 
Let $Y = [0,n]_{\Z}^v$.
Then the image
$(X \times Y, NP(\kappa,c_v))$ has the AFPP.
\end{thm}

\begin{proof}
We argue by induction on $v$. For $v=1$, the assertion
is correct by Theorem~\ref{extendToProd}.

Suppose, for some $k \in \N^*$, for $Y = \Pi_{i=1}^k [0,n]_{\Z}$,
$(X \times Y, NP(\kappa,c_k))$ has the AFPP.
Then by Theorem~\ref{extendToProd}, 
$(X \times Y \times [0,n]_{\Z}, NP(NP(\kappa,c_k)),c_1)$ 
has the AFPP. Note that
$X \times Y \times [0,n]_{\Z} = X \times [0,n]_{\Z}^{k+1}$,
and, by Lemma~\ref{NPlemma}, that
$NP(NP(\kappa,c_k),c_1) = NP(\kappa, c_{k+1})$.
This completes our induction.
\end{proof}

\begin{thm}
\label{arbCube}
Let $(X,\kappa)$ be a finite digital image with the AFPP. 
Then the image
$(X \times \Pi_{i=1}^v [a_i,b_i]_{\Z}, NP(\kappa,c_v))$
has the AFPP.
\end{thm}

\begin{proof}
This follows from Theorems~\ref{extendProdCube},
\ref{BEKLLisoThm}, and~\ref{BEKLLretractionThm}, 
as the image
\[(X \times \Pi_{i=1}^v [a_i,b_i]_{\Z}, NP(\kappa,c_v)) \]
is clearly isomorphic to an $NP(\kappa,c_v)$-retract 
of $X \times [0,n]^v$ for some $n$.
\end{proof}

\begin{thm}
\label{intervalAFPP}
{\rm \cite{Rosenfeld}}
The digital image $([a,b]_{\Z}, c_1)$ has 
the AFPP. 
\end{thm}

\begin{thm}
\label{containsUCube}
{\rm \cite{BxAFPP}}
Let $X \subset \Z^v$ be such that $X$ has a subset $Y = \Pi_{i=1}^v [a_i,b_i]_{\Z}$, where
$v>1$; for all indices $i$, $b_i \in \{a_i,a_i + 1\}$; and, for at least 2 indices $i$,
$b_i=a_i+1$. Then $(X,c_u)$ fails to have the AFPP
for $1 \le u < v$.
\end{thm}

As noted in~\cite{BxAFPP}, Theorem~\ref{containsUCube} states a severe limitation on the 
AFPP for digital images $X \subset \Z^v$ and the $c_u$ adjacency,
where $1 \le u < v$. We have the following.

\begin{thm}
For $1 \le u \le v$,
$(\Pi_{i=1}^v [a_i,b_i]_{\Z}, c_u)$ has the AFPP if and only
if $u=v$.
\end{thm}

\begin{proof}
For $u < v$, the assertion comes from 
Theorem~\ref{containsUCube}. Now consider the case $u=v$.
For $v=1$, the assertion follows from Theorem~\ref{intervalAFPP}.
For $v > 1$, Theorem~\ref{BoxKarNP} lets us conclude that
\[ (\Pi_{i=1}^v [a_i,b_i]_{\Z}, c_v) = 
   ([a_1,b_1]_{\Z} \times \Pi_{i=2}^v [a_i,b_i]_{\Z}, NP(c_1,c_{v-1})).
\]
The assertion follows from Theorem~\ref{arbCube}.
\end{proof}

\section{Further remarks}
We have shown that a digital image that is a tree has the AFPP.

A general answer to the question posed in the abstract is
not known at this writing. We have shown that given
a finite digital image $(X,\kappa)$ with the AFPP, then
$(X \times \Pi_{i=1}^v [a_i,b_i]_{\Z}, NP(\kappa,c_v))$
has the AFPP. It follows that $(\Pi_{i=1}^v [a_i,b_i]_{\Z},c_v)$
has the AFPP.

\bibliographystyle{amsplain}

\end{document}